\documentclass[12pt,
reqno]{amsart}

\usepackage[cp1251]{inputenc}
\usepackage{amsmath,amsxtra,amssymb,eufrak}%{amscd}
\usepackage[all]{xy}
\usepackage[active]{srcltx} %SRC Specials for DVI Searching
\usepackage{mathrsfs}
\usepackage[english]{babel}

\newtheorem{thm}{Theorem}[section]
\newtheorem{lm}[thm]{Lemma}

\newtheorem{pr}[thm]{Proposition}

\theoremstyle{definition}

\newtheorem{df}[thm]{Definition}
\newtheorem{exm}[thm]{Example}
\newtheorem{rem}[thm]{Remark}
\numberwithin{equation}{section}

\title{Length functions exponentially distorted on subgroups of complex Lie groups}

\author{O. Yu. Aristov}
\email{aristovoyu@inbox.ru}
\address {Institute for Advanced Study in Mathematics, Harbin Institute of Technology,  Harbin 150001, China}
\keywords{Length function, exponential distortion, complex Lie group, nilpotent radical, exponential radical, Banach PI-algebra, geometric group theory}
\subjclass[2020]{20F65, 22E30, Secondary 16R99}

\begin{document}

\begin{abstract}
We introduce a notion of a length function exponentially distorted  on a (compactly generated) subgroup of a locally compact group. We prove that for a connected linear complex Lie group there is a maximum equivalence class of length functions exponentially distorted on a normal integral subgroup lying between the exponential and nilpotent radicals. Moreover, a function in this class admits an asymptotic decomposition similar to that previously found by the author for word length functions, i.e., in the case of exponential radical [J. Lie Theory 29:4, 1045--1070, 2019]. In the general case we use auxiliary length functions constructed via holomorphic homomorphisms to Banach PI-algebras.
\end{abstract}

 \maketitle
\markright{Length functions}

The aim of this paper is to find  asymptotic decompositions (up to quasi-isometry) of some length functions (i.e., functions satisfying $\ell(gh)\le \ell(g)+\ell(h)$ for all $g,h$) on a connected linear complex Lie group. The study of such decompositions is motivated by the fact that they provide information on the structure of the Fr\'echet-Arens-Michael algebra associated with a length function; see~\cite[\S\,4]{AHHFG} and the follow-up paper~\cite{ArPC15}. It was  Akbarov who noted that word length functions (which are known to be equivalent to those associated with a left invariant Riemannian metric; see, e.g., \cite[Proposition 5.5]{ArAMN}) can be used for describing the structure of the Arens-Michael envelope of the algebra of analytic functionals on the group \cite[\S\,5.2]{Ak08}. (In fact, he used the exponentials of length functions, which are called semicharacters or, in other terminology, submultiplicative weights.)  An asymptotic decomposition of a word length function is obtained in~\cite{ArAnF}. Here we want to prove more.

In addition to word length functions, a Lie group can admit many others. But which ones are interesting in some way? In this article we present  a lot of `interesting' length functions (for a 'sufficiently large' Lie group), which are the solutions of certain extremal problems. Specifically, we introduce a notion of length function \emph{exponentially distorted} on a given subgroup and prove that for every normal integral subgroup lying between the exponential and nilpotent radicals there is a maximal (up to equivalence) length function of this type. This is the first assertion in our main result, Theorem~\ref{maindisto}. The second gives an asymptotic decomposition of such a length function. The proof is based on a technique developed in~\cite{ArAnF} and the result itself is a generalization of a theorem in those paper.

The idea to consider  exponential distortion on subgroups of a complex Lie group comes for the following observation. It is proved  in \cite{ArPiLie}  that the range of a Lie-algebra homomorphism to a Banach algebra satisfies a polynomial identity if and only if a certain polynomial growth condition holds on the image of the nilpotent radical of the Lie algebra (see Theorem~\ref{forHom} below). In the Lie-group context (Theorem~\ref{exdiNR}) this condition means that the length function associated with a holomorphic homomorphism to a Banach algebra has exponential distortion on the nilpotent radical, now of the Lie group.  On the other hand,  it is known that a word length function has exponential distortion on the so-called exponential radical (which is smaller) \cite{Co08}. Thus it is natural to look on subgroups intermediate between two radicals. And we do it.

Here we consider only complex Lie groups.
There is no doubt that a similar theory can be developed for real Lie groups (see some hints in Remark~\ref{nonscreal}). But our goal is applications to completions of algebras of analytic functionals in the complex case. So we confine ourselves to this case.

Note also that the length functions constructed in this article corresponds to a family of equivalence classes (up to quasi-isometry) of metric spaces parameterized by integral subgroups under consideration (Remark~\ref{nqifam}). It seems likely that these metric spaces is of interest from the point of view of geometric groups theory.

This paper is the third in a series of five dealing with properties of completions of algebras of analytic functionals on connected complex Lie groups. The structure of the series is as follows.  The completions themselves are studied in the fourth and fifth papers: the theorem on weight decomposition proved here is used in \cite{ArPC15} to show that the corresponding completion can be decomposed into an iterated analytic smash product, and in \cite{ArFin} smash products are used in proving that the completion homomorphism is a homological epimorphism. On the other hand, this article are essentially based  on \cite{ArPiLie} (the second article). The results in the first article \cite{ArLine} are needed only for \cite{ArFin}.

The main text of this paper is divided into two sections. In \S\,\ref{sec:exdi} we collect some basic facts on radicals of a Lie group and length functions, introduce  a definition of exponential distortion, consider some examples, and describe a construction of a length function exponentially distorted on a subgroup contained in the nilpotent radical (Proposition~\ref{omegpr0}). Banach PI-algebras play an important role in this consideration. Our main result, Theorem~\ref{maindisto}, are stated and proved in~\S\,\ref{sec:decth}. We show the existence of a maximal  exponentially distorted length function and derive a formula for an asymptotic decomposition. It is also proved there that the equivalence class of such a function is uniquely determined by the corresponding subgroup (Proposition~\ref{N1N2}).

\section{Exponentially distorted length functions}
\label{sec:exdi}

In this section we have compiled preliminary information on the nilpotent radical and the exponential radical. We also define exponential distortion and give a construction of a length function exponentially distorted on an integral subgroup contained in the nilpotent radical.

\subsection*{Nilpotent radical}

In this article we consider finite-dimensional complex Lie algebras and Lie groups.

Recall that the \emph{nilpotent radical} $\mathfrak{n}$ of a Lie algebra $\mathfrak{g}$ is the intersection of the kernels of all finite-dimensional irreducible representations of~$\mathfrak{g}$ (\cite[Chapter~I, \S\,5.3, p.\,44, Definition~3]{Bou} or \cite[p.\,27, \S\,1.7.2]{Dix}). Similarly, the \emph{nilpotent radical} $N$  of a complex Lie group~$G$ is the intersection of the kernels of all finite-dimensional irreducible holomorphic representations of~$G$. We follow the terminology in \cite{Dix,OV3};  in \cite{Le02} the term `representation radical' is used. It is obvious that $N$ is a closed normal subgroup of~$G$.

Let $\mathfrak{g}$ be the Lie algebra of~$G$. Recall that the subgroup of $G$ generated by  $\exp\mathfrak{h}$, where $\mathfrak{h}$ is a Lie subalgebra of $\mathfrak{g}$ is said to be \emph{integral}. The \emph{solvable radical} of a connected Lie group is, by definition, the integral subgroup corresponding to the solvable radical of~$\mathfrak{g}$. Also, if $H_1$ and $H_2$ are subgroups of $G$, then we denote by $(H_1,H_2)$ the subgroup generated by the multiplicative commutators $h_1 h_2 h_1^{-1}h_2^{-1}$, where $h_1\in H_1$ and $h_2\in H_2$. Recall that a complex Lie group is said to be \emph{linear} if it admits a faithful holomorphic representation.

The following proposition is more or less standard.

\begin{pr}\label{NRprop}
If a complex Lie group~$G$ is connected and linear, then the nilpotent radical $N$ is a nilpotent, simply connected, integral subgroup with the nilpotent radical $\mathfrak{n}$ as its Lie algebra. Also, $N$ coincides with $(G,R)$ and the solvable radical of $(G,G)$, where $R$ is the solvable radical of $G$.
\end{pr}
\begin{proof}
Note that $N$ is nilpotent and simply connected by \cite[Theorem 4.38 and Corollary 4.39]{Le02} (warning: the definition of a `complex analytic group' in \cite{Le02} includes the assumption of being connected). Moreover, $N$ coincides with the radical of $(G,G)$ by \cite[Corollary 4.39]{Le02}. In particular, $N$ is integral.

Note that $\mathfrak{n}=[\mathfrak{g},\mathfrak{r}]=[\mathfrak{g},\mathfrak{g}]\cap\mathfrak{r}$, where $\mathfrak{r}$ is the solvable radical of $\mathfrak{g}$ \cite[p.\,26, Proposition 1.7.1]{Dix}. It is easy to see that the solvable radical of $[\mathfrak{g},\mathfrak{g}]$ coincides with $[\mathfrak{g},\mathfrak{g}]\cap \mathfrak{r}$ and therefore with $\mathfrak{n}$. Thus, $\mathfrak{n}$ is the Lie algebra of~$N$.

Since $G$ and $R$ are the integral subgroups of $G$ corresponding to $\mathfrak{g}$ and $\mathfrak{r}$,  and $\mathfrak{n}$ is an ideal, it follows that $(G,R)$ is the integral subgroup corresponding to $[\mathfrak{g},\mathfrak{r}]$ \cite[pp.\,443--444, Lemma 11.2.2 and Proposition 11.2.3]{HiNe}. Then $[\mathfrak{g},\mathfrak{r}]=\mathfrak{n}$ implies that $(G,R)=N$.
\end{proof}

Thus, $N$ can  also be defined as the integral subgroup corresponding to the nilpotent radical of~$\mathfrak{g}$.

\begin{lm}\label{closc}
Let $G$  be a connected linear complex Lie group and $N$ the nilpotent radical of $G$. If $N'$ is an arbitrary integral subgroup of $N$, then it is nilpotent, closed and simply connected.
\end{lm}
\begin{proof}
Every connected linear complex group Lie can be represented as a semidirect product $B\rtimes L$, where $B$ is simply connected solvable  and $L$ is linearly complex reductive; see \cite[p.\,601, Theorem 16.3.7]{HiNe}.  Note that $N\subset B$ (see \cite[Proposition 4.44]{Le02}, where $B$ is called a `nucleus'). Therefore $N'\subset B$. Being an integral subgroup of a simply connected solvable Lie group,  $N'$ is closed and simply connected \cite[p.\,449, Proposition 11.2.15]{HiNe}. Finally, $N'$ is nilpotent because it is contained in  $N$.
\end{proof}

\subsection*{Exponential radical}

Let $E$ denote the \emph{exponential radical} of a connected complex Lie group $G$. (The terminology is from~\cite{Os02}; the definition for a real connected Lie group is given in~\cite[\S\,6]{Co08}.) For the complex case we
use the following characterization as a definition:
$E$ is the normal complex Lie subgroup such that $G/E$ is
the largest quotient of $G$ that is locally isomorphic to a direct
product of a nilpotent  and  semisimple complex Lie group \cite[Proposition 3.11]{ArAnF}.

Let
\begin{equation}\label{fedef}
\mathfrak{e}\!: = \mathfrak{r}_\infty + (\mathfrak{s},\mathfrak{r}),
\end{equation}
where $\mathfrak{r}$ is the solvable radical of $\mathfrak{g}$, $\mathfrak{r}_\infty$ is the intersection of its lower central series, $\mathfrak{s}$ is a Levi complement and $(\mathfrak{s},\mathfrak{r})$ is the Lie subalgebra generated by $[\mathfrak{s},\mathfrak{r}]$; see \cite[Formula~(4)]{ArAnF} and \cite[Remark 4.8]{AHHFG}.

\begin{pr}\label{exprpr}
If a complex Lie group~$G$ is connected and linear, then the exponential radical $E$ is a  nilpotent, simply connected, integral subgroup with $\mathfrak{e}$ as its Lie algebra. Also, $E\subset N$.
\end{pr}
\begin{proof}
By \cite[Proposition 3.13]{ArAnF}, $E$ is simply connected and coincides with $\exp\mathfrak{e}$. In particular, $E$ is integral. Since $\mathfrak{n}=[\mathfrak{g},\mathfrak{g}]\cap\mathfrak{r}$ (see the proof of Proposition~\ref{NRprop}), we have $\mathfrak{e}\subset\mathfrak{n}$ and hence $E\subset N$. Since $N$ is nilpotent, so is~$E$.
\end{proof}

\subsection*{Exponential distortion}

Recall that a \emph{length function} on a locally compact group $G$ is a non-negative  function $\ell\!: G \to\mathbb{R}$ satisfying $\ell(gh)\le \ell(g)+\ell(h)$ for all $g, h \in G$. We always assume that a length function is locally bounded. A length function is said to be \emph{symmetric} if $\ell(g^{-1})=\ell(g)$ and $\ell(1)=0$.

The basic example is a word length function. Suppose that $G$ is compactly
generated, i.e.,  it is generated by a relatively compact subset $U$ that is a neighbourhood of the identity.
Then
\begin{equation}\label{wordlen}
\ell_U(g)\!: = \min \{ n \!: \, g \in U^{n} \},
\end{equation}
(where $U^0=\{1\}$) is a length function, which is called a \emph{word length function}.

Let $\varphi_1$ and $\varphi_2$ be non-negative functions on a set~$X$. We say that $\varphi_1$ is \emph{dominated} by $\varphi_2$ and write $\varphi_1\lesssim \varphi_2$ if there are $C,D>0$ such that
$$
\varphi_1(x)\le C\,\varphi_2(x)+D \quad\text{for all $x\in X$.}
$$
If $\varphi_1\lesssim \varphi_2$ and $\varphi_2\lesssim \varphi_1$, we say that $\varphi_1$ and $\varphi_2$ are \emph{equivalent} and write
$\varphi_1\simeq\varphi_2$.

Note that two length functions  on a locally compact group are equivalent if and only if the corresponding pre-metric spaces are quasi-isometric. In fact,  this kind of equivalence is rather weak (but there are weaker ones). In particular, it cannot see the difference between the length functions induced by a left-invariant Riemannian metric and a Carnot-Carath\'{e}odory metric (which is sub-Riemannian) on a simply connected nilpotent Lie group; see \cite[Appendix]{ArAMN}.

\begin{df}
Let $G$ be a locally compact group, $H$ a compactly generated closed subgroup and $\tilde\ell$ a word length function on~$H$.
We say that a length function $\ell$ is \emph{exponentially distorted} on $H$ if
$$
\ell \lesssim  \log(1+ \tilde{\ell}\,) \quad\text{on $H$}
$$
and \emph{strictly exponentially distorted} on $H$ if
$$
\ell\simeq  \log(1+ \tilde{\ell}\,)\quad\text{on $H$.}
$$
\end{df}

In the case when $G$ is also compactly generated and $\ell$ is a word length function, this definition gives the well-known (and well-studied) notions of a \emph{exponentially distorted subgroup} and  \emph{strictly exponentially distorted subgroup}. (Traditionally, a symmetric word length function is used in the definition but in view of Lemma~\ref{woslf} below this is irrelevant.)

The motivation for introducing these notions comes form the fact that the length function associated with a holomorphic homomorphism to a Banach PI-algebra has the same behaviour on the nilpotent radical (see Theorems~\ref{forHom} and~\ref{exdiNR}) as a word length function on the exponential radical. The latter is described in the following result.

\begin{thm}\label{expRdis}
Let $G$  be a connected complex Lie group and $E$ denote the exponential radical of $G$.

\emph{(A)} Every length function on $G$ is exponentially distorted on $E$.

\emph{(B)}  Every word length function  on $G$ is strictly exponentially distorted on $E$.
\end{thm}
\begin{proof}
The Part~(B) is a partial case of \cite[Theorem 6.5(v)]{Co08} and Part~(A) easily follows Part~(B).
\end{proof}

We want to describe both cases, of the nilpotent and exponential radical, in unified terms that can also be applied to intermediate cases. Our first goal is to construct length functions strictly exponentially distorted on some subgroups of~$N$.

In the following two lemmas we collect some operations that enable us to obtain new length functions from given ones.

\begin{lm}\label{safes}
Let $\varphi_1$ and $\varphi_2$ be non-negative functions on a set~$X$ and $\varphi_1\lesssim \varphi_2$. If
$f\!:\mathbb{R}_+\to\mathbb{R}_+$ is a subadditive increasing  function, then $f\circ\varphi_1\lesssim f\circ\varphi_2$.
\end{lm}
\begin{proof}
Take $C,D>0$ such that $\varphi_1(x)\le C\varphi_2(x)+D $ for every $x\in X$. Take $n\in\mathbb{N}$ such that $n\ge C$. Then
$$f(\varphi_1(x))\le  f(n\varphi_2(x)+D)\le n f(\varphi_2(x))+f(D)$$ for every $x\in X$.
\end{proof}

\begin{lm}\label{3oper}
Let $G$ be  a locally compact group.

\emph{(1)}~If $\ell$ is a length function on $G$ and $f\!:\mathbb{R}_+\to\mathbb{R}_+$ is a subadditive increasing continuous function such that $f(0)=0$, then $f\ell$ is a length function on $G$.

\emph{(2)}~If $\ell'$ and $\ell''$ are length functions on $G$, then  $\ell'+\ell''$ \emph{(}and $\max\{\ell',\,\ell''\}$, which is equivalent\emph{)}   is a length function on $G$.

\emph{(3)}~If $H$ is a closed normal subgroup of $G$, $\sigma\!:G\to G/H$ denotes the quotient homomorphism and $\ell$ is a length function on $G/H$, then $\ell\sigma$ is a length function on $G$.
\end{lm}

The proof is straightforward.

\begin{lm}\label{woslf}
A word length function on a compactly generated locally compact group is equivalent to a symmetric word length function.
\end{lm}
\begin{proof}
Let $U$ be a relatively compact generating neighbourhood of the identity and $V=U\cup U^{-1}$.   Obviously $\ell_V$ is symmetric and $\ell_V\le \ell_U$. Since the product of finitely many open subsets of a topological group is open and $V$ is relatively compact,  there is $n\in\mathbb{N}$ such $V\subset U^n$. Therefore $\ell_U\lesssim \ell_V$.
\end{proof}

We also need the symmetric length function
$$
\ell^{sym}(g)\!:=\ell(g)+\ell(g^{-1})
$$
associated with a given length function~$\ell$.

\begin{lm}\label{symexd}
Let $G$ be a locally compact group and $H$ a compactly generated closed subgroup.
A length function $\ell$ on $G$ is \emph{(}strictly\emph{)} exponentially distorted on $H$ if and only if $\ell^{sym}$ is \emph{(}strictly\emph{)} exponentially distorted on $H$.
\end{lm}
\begin{proof}
Consider the length function $\ell^{op}(g)\!:=\ell(g^{-1})$. By the definition, $\ell$ is exponentially distorted on $H$ if $\ell\lesssim\log(1+ \tilde{\ell})$ on $H$ for some  word length function $\tilde\ell$ on~$H$. It follows from Lemma~\ref{woslf} that we can assume that $\tilde{\ell}$ is symmetric. Then the conditions that $\ell^{op}\lesssim\log(1+ \tilde{\ell})$ and $\ell^{sym}\lesssim\log(1+ \tilde{\ell})$ on $H$ are both equivalent to the condition that $\ell\lesssim\log(1+ \tilde{\ell})$ on $H$.

The case of strict exponential distortion is similar.
\end{proof}

The main goal of this section is to construct a sufficiently big length function exponentially distorted on  an integral subgroup contained in the nilpotent radical. The three operations in Lemma~\ref{3oper} are very useful (see Example~\ref{Heisex}) but still insufficient to our purposes (see Example~\ref{sixdim}). We need another well-known construction.

Let $A$ be a Banach algebra with norm $\|\cdot\|$. Then every (strongly) continuous homomorphism  $\pi\!:G\to {\mathop{\mathrm{GL}}\nolimits}(A)$, where ${\mathop{\mathrm{GL}}\nolimits}(A)$ is the group of invertible elements of~$A$,  induces the length functions given by
\begin{equation}\label{lfhom}
\ell_\pi(g)\!:=\log\|\pi(g)\|\quad\text{and}\quad \ell_\pi^{sym}(g)=\log\|\pi(g)\|+\log\|\pi(g^{-1})\|\qquad(g\in G).
\end{equation}
($\ell_\pi$ is locally bounded by the uniform boundedness principle.)

Now we consider examples.

\begin{exm}\label{Heisex}
Let $H$ be the $3$-dimensional complex Heisenberg group, i.e., the group of matrices of the form
\begin{equation}\label{3Hei}
h=
\begin{pmatrix}
 1& a& c\\
 0&1 & b\\
 0 & 0&1
\end{pmatrix}\qquad
(a,b,c\in\mathbb{C}).
\end{equation}
It is well known that every word length function on $H$ is equivalent to $|a|+|b|+|c|^{1/2}$. The nilpotent radical $N$ coincides with the centre of $H$ (the $1$-dimensional subgroup of matrices such that $a=b=0$). Here we describe two methods for constructing a length function strictly exponentially distorted on $N$.

(1)~We can easily do this using the operations in Lemma~\ref{3oper}. Indeed, taking the composition of a word length function with the subadditive increasing function $x\mapsto\log(1+x)$, we define a length function $\ell'$ that is equivalent to $\log(1+|a|+|b|+|c|)$ by Lemma~\ref{safes}. (Note that the  exponent $1/2$ becomes irrelevant after applying the logarithm.) Thus, $\ell'$ is strictly exponentially distorted on the whole $H$. We can increase it so that it becomes strictly exponentially distorted only on $N$ but not on the whole~$H$.  Denote by $\ell''$ the length function $h\mapsto|a|+|b|$ on $H$. It is a lift of a word length function on $H/N$. Then $\ell'+\ell''$ is equivalent to  $|a|+|b|+\log(1+|c|)$ and so it is strictly exponentially distorted on $N$. It is also not hard  to see that every length function exponentially distorted on $N$ is dominated by  $\ell'+\ell''$.

Note that the method can also be  used  to obtain length functions with different types of distortion, e.g., equivalent to  $|a|^\alpha+|b|^\beta+|c|^\gamma$, where $\alpha,\beta \in(0,1]$ and $\gamma\in(0,1/2]$.

(2)~The above construction does not work when the group is more complicated (see Example~\ref{sixdim}). The another method that can be applied in the general case is as follows. Denote by $\pi$ the standard representation of $H$ (treating the matrix in~\eqref{3Hei} as an operator) and  a submultiplicative norm $\|\cdot\|$ on the algebra of $3\times 3$ matrices. Then the length function $\ell_\pi$ given by~\eqref{lfhom} is equivalent to $\log(1+|a|+|b|+|c|)$.  Thus we can replace $\ell'+\ell''$ considered above by $\ell_\pi+\ell''$. For a general form of this construction see Proposition~\ref{omegpr0}.
\end{exm}

\begin{exm}\label{sixdim}
Let $\mathfrak{g}$ be the $6$-dimensional solvable Lie algebra in \cite[Example 5.14]{ArAnF}. We do not need the whole definition and use only the fact that both $\mathfrak{g}/\mathfrak{e}$ and $\mathfrak{e}$ are isomorphic to the $3$-dimensional Heisenberg Lie algebra.
Denote by $G$ the associated simply connected Lie group. Consider the coordinates $(s_1, s_2, s_3, t_1, t_2, t_3)$ on $G$, corresponding to the canonical coordinates of the first kind on both copies of the Heisenberg algebra (first three for $\mathfrak{g}/\mathfrak{e}$ and last three for $\mathfrak{e}$; see details in the above reference). Then each word length function on $G$ is equivalent to $$|s_1|+|s_2|+|s_3|^{1/2}+\log(1+|t_1|+|t_2|+|t_3|).$$ (This is follows from \cite[Theorem 4.1]{ArAnF} and  can also be obtained from Theorem~\ref{maindisto} below). Note that the nilpotent radical $N$ is determined by the condition $s_1=s_2=0$ and the exponential radical by the condition $s_1=s_2=s_3=0$.

If we directly apply the function $x\mapsto\log(1+x)$ as in Example~\ref{Heisex}, then we get the double logarithm of $t_1, t_2, t_3$. This corresponds to the superexponential distortion on $E$ and the resulting length function is not  strictly exponential distorted on $N$. So the first method in Example~\ref{Heisex} does not work properly. But the second method is useful for making distortion strictly exponential and we confirm this claim below in Proposition~\ref{omegpr0}.
\end{exm}

\begin{exm}
Note that a word length function can also be of the form~\eqref{lfhom} (up to equivalence). Indeed, consider the following $2$-dimensional
representation of the group~$\mathbb{C}$:
$$
\pi(z)=\begin{pmatrix}    e^z &  0\\
0&  e^{iz}\end{pmatrix}.
$$
Let $\ell_\pi$ and $\ell_\pi^{sym}$ be as in \eqref{lfhom}. Since
$$
\max\{|e^z|,\,|e^{-z}|,\, |e^{iz}|,\,|e^{-iz}|\}=\max\{|e^z|,\,|e^{iz}|\}=\max\{e^{|x|},\, e^{|y|}\},
$$
where $z=x+iy$ with $x,y\in\mathbb{R}$,
we have
$$
\ell_\pi(z)\simeq
\ell_\pi^{sym}(z)\simeq \max\{|x|,\, |y|\}.
$$
So it is equivalent to the standard length function $|z|=\sqrt{x^2+y^2}$, which in turn is equivalent to a word length function.
\end{exm}

\subsection*{PI-completions}
Recall that an associative algebra $A$ over $\mathbb{C}$ is called a PI-\emph{algebra} (satisfies a \emph{polynomial identity}) if there is a non-trivial non-commutative polynomial $p$ (i.e., an element of a free algebra in $n$ generators) such that  $p(a_1,\ldots,a_n)=0$ for all $a_1,\ldots,a_n\in A$. It is important for us that  commutative and finite-dimensional algebras satisfy PI. For general PI-algebras see, e.g., \cite{AGPR} or \cite{KKR16}; for Banach PI-algebras see \cite{Kr87} or \cite{Mu94}.

Our study is essentially based on the following result, which gives necessary and sufficient conditions for a completion of the universal enveloping algebra of a Lie algebra to satisfy a PI. We denote by $U(\mathfrak{g})$ the universal enveloping algebra of a Lie algebra $\mathfrak{g}$.

\begin{thm}\label{forHom} \cite[Theorem 2]{ArPiLie}
Suppose that $\mathfrak{g}$ is a finite-dimensional complex Lie algebra, $A$ is a Banach algebra and $\theta \!:U(\mathfrak{g})\to A$ is a homomorphism of associative algebras. Let $\|\cdot\|$ be the norm on~$A$ and $|\cdot|$ some norm on~$\mathfrak{n}$. If $\theta $ has dense range, then the following conditions are equivalent.

\emph{(1)}~$A$ is a \emph{PI}-algebra.

\emph{(2)}~$\theta (\eta)$ is nilpotent for every  $\eta\in \mathfrak{n}$.

\emph{(3)}~$e^{\theta(\eta)}-1$ is nilpotent for every  $\eta\in \mathfrak{n}$.

\emph{(4)}~There are $C>0$ and $\alpha>0$ such that $\|e^{\theta (\eta)}\|\le C (1+|\eta|)^\alpha$ for every $\eta\in\mathfrak{n}$.
\end{thm}

We now transfer this result to the context of Lie groups.
Let $A$ be a Banach algebra  and $\pi\!:G\to {\mathop{\mathrm{GL}}\nolimits}(A)$ a holomorphic homomorphisms. (Here ${\mathop{\mathrm{GL}}\nolimits}(A)$ is the group of invertible elements of~$A$, which can be treated as a Banach-Lie group; see, e.g.,  \cite[Example III.1.11(b)]{Ne05}).
We  can identify the Lie algebra of ${\mathop{\mathrm{GL}}\nolimits}(A)$  with $A$ \cite[Chapter~3, \S\,3.9]{Bou}.  Applying the Lie functor to $\pi$ we get a Banach-Lie algebra homomorphism $\mathrm{L}_\pi\!: \mathfrak{g} \to A $ such that $\pi\exp =\exp \mathrm{L}_\pi$, where  $\exp$ denotes the exponential maps $\mathfrak{g}\to A$ and $A\to {\mathop{\mathrm{GL}}\nolimits}(A)$.

\begin{thm}\label{exdiNR}
Let $G$  be a connected linear complex Lie group and $N$ the nilpotent radical of $G$. Suppose that $A$ is a Banach algebra and  $\pi\!:G\to {\mathop{\mathrm{GL}}\nolimits}(A)$ is a holomorphic homomorphism. Consider the following conditions.

\emph{(1)}~$A$ is a \emph{PI}-algebra.

\emph{(2)}~The length function $\ell_\pi(g)=\log\|\pi(g)\|$  \emph{(}see~\eqref{lfhom}\emph{)} is exponentially distorted on $N$.

Then $(1)\Rightarrow(2)$. If, in addition, the homomorphism $\theta \!:U(\mathfrak{g})\to A$ induced by $\mathrm{L}_\pi\!: \mathfrak{g} \to A$ has dense range, then $(1)$ is equivalent to $(2)$.
\end{thm}

For the proof we need the following result.
\begin{pr}\label{charexd}
Let $H$ be a simply connected nilpotent closed subgroup of a \emph{(}complex or real\emph{)} Lie group~$G$, $\mathfrak{h}$~its Lie algebra, $|\cdot|$ a norm on~$\mathfrak{h}$ and $\ell$ a length function  on~$G$. Then

\emph{(A)}~$\ell$ is exponentially distorted on $H$ if and only if $\ell(\exp(\eta))\lesssim \log(1+|\eta|)$ on  $\mathfrak{h}$;

\emph{(B)}~$\ell$ is strictly exponentially distorted on $H$ if and only if $\ell(\exp(\eta))\simeq \log(1+|\eta|)$ on~$\mathfrak{h}$.
\end{pr}
\begin{proof}
Let $\tilde\ell$ be a word length function on~$H$. Since $H$ is simply connected nilpotent,
\begin{equation}\label{loffornilp}
\log(1+|\eta|) \simeq \log(1+ \tilde\ell(\exp(\eta))\quad \text{on~$\mathfrak{h}$}
\end{equation}
(this fact is well known; see, e.g., \cite[Lemma~4.7]{ArAnF}). The rest is clear.
\end{proof}

\begin{proof}[Proof of Theorem~\ref{exdiNR}]
Since $\pi\exp =\exp \mathrm{L}_\pi$, we have
\begin{equation}\label{piexpeta}
\pi(\exp(\eta))=e^{\mathrm{L}_\pi(\eta)}=e^{\theta (\eta)}.
\end{equation}
Denote by $A_0$ the closure of the range of $\theta $.
By Theorem~\ref{forHom},  $A_0$ is a PI-algebra if and only if there are $C>0$ and $\alpha>0$ such that $\|\pi(\exp(\eta))\|\le C (1+|\eta|)^\alpha$ for every $\eta\in\mathfrak{n}$. But the last means that $\log\|\pi(\exp(\eta))\|\lesssim  \log(1+|\eta|)$ on  $\mathfrak{n}$ and it follows from Part~(A) of Proposition~\ref{charexd} that this is equivalent to the condition that  $\ell_\pi$ is exponentially distorted on $N$ (by Proposition~\ref{NRprop}, $N$ is simply connected and nilpotent because $G$ is linear).

If $A$ is a PI-algebra, then so is $A_0$. Thus we have that
$(1)\Rightarrow(2)$. When the range of $\theta $ is dense, $A_0=A$. In this case we obtain $(1)\Leftrightarrow(2)$.
\end{proof}

The following proposition shows how to increase an exponentially distorted  length function to make it strictly exponentially distorted.

\begin{pr}\label{omegpr0}
Let $G$ be a connected complex Lie group, $N'$ an integral subgroup of $G$ contained in $N$ and $\ell'$ a length function exponentially distorted on $N'$. Suppose that $A$ is a Banach \emph{PI}-algebra and  $\pi\!:G\to {\mathop{\mathrm{GL}}\nolimits}(A)$ is an injective holomorphic homomorphism. Then $\ell'+ \ell_\pi^{sym}$ \emph{(}where  $\ell_\pi^{sym}$ is as in~\eqref{lfhom}\emph{)} is strictly exponentially distorted on $N'$.
\end{pr}
\begin{proof}
Since $G$ admits an injective holomorphic homomorphism to the group of invertible elements of a Banach algebra, it is linear \cite[Theorem 2.2]{ArAnF}.  We have by Theorem~\ref{exdiNR} that $\ell_\pi$  is exponentially distorted on the nilpotent radical $N$ and so is $\ell_\pi^{sym}$  by Lemma~\ref{symexd}.  By the hypothesis, $\ell'$ is exponentially distorted on $N'$ and hence so is $\ell'+ \ell_\pi^{sym}$.

To prove that the distortion of $\ell'+ \ell_\pi^{sym}$ is strictly exponential it suffices to show that  the distortion of $\ell_\pi^{sym}$ is strictly exponential. Note that by Lemma~\ref{closc}, $N'$ is nilpotent, closed and simply connected and so we can apply Proposition~\ref{charexd}. Thus all we need to show is that
$$
\log(1+|\eta|)\lesssim \ell_\pi^{sym}(\exp(\eta))\quad \text{on~$\mathfrak{n}'$}.
$$
We prove more: this holds on the whole $\mathfrak{n}$.

Since $\pi$ is injective, the map  between the algebras of holomorphic functions on  ${\mathop{\mathrm{GL}}\nolimits}(A)$ and $G$ has dense range.  The elements of the corresponding Lie algebras can be identified with continuous $\mathbb{C}$-valued derivations on these algebras. The density of the range implies that $\mathrm{L}_\pi$ is also injective.

Thus $|\eta|\!:= \|\mathrm{L}_\pi(\eta)\|$ is a norm on $\mathfrak{n}$. Fix $\eta\in\mathfrak{n}$ and  let $a\!:= \mathrm{L}_\pi(\eta)$. Since $A$ is a PI-algebra, so is the closure of the range of $\theta \!:U(\mathfrak{g})\to A$. Applying
Theorem~\ref{forHom} to this closure we conclude that $a$ is nilpotent. It can be easily seen that there is $m\in\mathbb{Z}_+$ such that $1,a,\ldots,a^{m-1}$ are linearly independent and $a^m=0$.

We can assume that $\eta\ne 0$.
Then $a\ne 0$ by the injectivity. So by the linear independence and the Hahn-Banach theorem, there is a bounded linear functional $f$ on $A$ such that $\langle f, 1 \rangle=0$, $\langle f,a\rangle=\|a\|$  and $\langle f,a^k\rangle=0$ when $k>1$.
Hence,
$$
\langle f, e^{a}\rangle= \sum_{k=0}^\infty \frac{\langle f,a^k\rangle}{k!}=\|a\|.
$$
Then for every $\eta\in\mathfrak{n}$ we have $|\eta|=\|a\|\le \|f\|\,\|e^{a}\|=\|f\|\,\|\pi (\exp(\eta))\|$  (the last equality by \eqref{piexpeta}) and therefore
$$
|\eta|\le \|f\|\, \exp(\ell_\pi(\exp(\eta)))\le \|f\|\, \exp(\ell_\pi^{sym}(\exp(\eta))).
$$
On the other hand, it follows from the definition of $\ell_\pi^{sym}$ that $1\le\exp(\ell_\pi^{sym}(g))$ for every $g$ (since $1=\|\pi(1)\|\le \|\pi(g)\|\,\|\pi(g^{-1})\|$). Combining the bounds for $1$ and $|\eta|$ we have that
$$
1+|\eta|\le  (1+\|f\|)\, \exp(\ell_\pi^{sym}(\exp(\eta)))
$$
and so $\log(1+|\eta|)\le  \ell_\pi^{sym}(\exp(\eta))+\log (1+\|f\|)$ on $\mathfrak{n}$. Thus $\ell'+ \ell_\pi^{sym}$ is strictly exponentially distorted on $N'$.
\end{proof}

\section{A decomposition theorem for length functions}
\label{sec:decth}

A decomposition theorem for a word length function is proved in \cite{ArAnF}.
It corresponds to the case of  the exponential radical. With some modification we can produce a proof that also works  for a subgroup intermediate between two radicals, exponential  and nilpotent.

\begin{pr}\label{expVN}
Let $B$ be a simply connected solvable complex Lie group, $\mathfrak{b}$ its Lie algebra, and $\mathfrak{n}'$ a nilpotent subalgebra of~$\mathfrak{b}$ such that $\mathfrak{b}/\mathfrak{n}'$ is also nilpotent. Suppose that~$\mathfrak{h}$ is a nilpotent subalgebra of~$\mathfrak{b}$ such that $\mathfrak{n}'+\mathfrak{h}= \mathfrak{b}$ and $\mathfrak{v}$ is a subspace of~$\mathfrak{h}$ complementary to $\mathfrak{n}'\cap \mathfrak{h}$.
Then the map
$$
 \mathfrak{n}'\times\mathfrak{v}\to B\!:(\eta,\xi)\mapsto \exp(\eta)\exp(\xi)
$$
is a biholomorphic equivalence of complex manifolds.
\end{pr}
This result is proved in \cite[Proposition 4.8]{ArAnF} in the case when $B$ and the subalgebras of~$\mathfrak{b}$ are specified. We omit the proof since the argument still works in general because it uses only the assumptions that $\mathfrak{n}'+\mathfrak{h} = \mathfrak{b}$ and $\mathfrak{h}$, $\mathfrak{n}'$,  $\mathfrak{b}/\mathfrak{n}'$ are nilpotent.

Now we specify  the subalgebras and subspace in Proposition~\ref{expVN} in a more general context than in \cite{ArAnF}.
A connected linear Lie complex group $G$ can be represented as a semidirect product
$B\rtimes L$, where $B$ is simply connected solvable  and $L$ is linearly complex reductive; see \cite[p.\,601, Theorem 16.3.7]{HiNe} or \cite[Theorem 4.43]{Le02}. Fix such a decomposition. Recall that $N\subset B$; see the proof of Lemma~\ref{closc}.

Let $N'$ be a normal integral subgroup of $G$ such that $E\subset N'\subset N$. Now
\begin{itemize}
  \item let~$\mathfrak{b}$ and $\mathfrak{n}'$ be the Lie algebras of $B$ and $N'$, respectively;
  \item let $\mathfrak{h}$ be a Cartan subalgebra in $\mathfrak{b}$;
  \item let $\mathfrak{v}$ be a subspace of~$\mathfrak{h}$ complementary to $\mathfrak{n}'\cap \mathfrak{h}$.
\end{itemize}

\begin{thm}\label{expVNL}
Let $G$ be a connected linear complex Lie group and let $\mathfrak{n}'$, $\mathfrak{v}$ and $L$ be as above. Then the map
$$
\mathfrak{n}'\times \mathfrak{v}\times L\to G\!:(\eta,\xi,l)\mapsto
\exp(\eta)\exp(\xi)\,l
$$
is a biholomorphic equivalence of complex manifolds.
\end{thm}
\begin{proof}
Since $G$ is a semidirect product of $B$ and $L$, it suffices to show that the hypotheses in Proposition~\ref{expVN} hold.

Since $\mathfrak{h}$ is a Cartan subalgebra, $\mathfrak{h}+\mathfrak{b}_\infty = \mathfrak{b}$, where $\mathfrak{b}_\infty$ is the intersection of the lower central series of $\mathfrak{b}$ \cite[Chapter~7, \S\,2, no.\,1, Corollary~3]{Bou2}. Since $\mathfrak{b}$ is solvable, it is contained in
the solvable radical $\mathfrak{r}$ of $\mathfrak{g}$. Therefore $\mathfrak{b}_\infty$ is contained in $\mathfrak{r}_\infty$, the intersection of the lower central series of $\mathfrak{r}$. Further, $\mathfrak{r}_\infty\subset\mathfrak{e}$; see~\eqref{wordlen}. Since in turn $\mathfrak{e}\subset \mathfrak{n}'$, we have $\mathfrak{n}'+\mathfrak{h} = \mathfrak{b}$.

Finally, $\mathfrak{h}$, $\mathfrak{n}'$ and  $\mathfrak{b}/\mathfrak{n}'$ are nilpotent: $\mathfrak{h}$ being a Cartan subalgebra, $\mathfrak{n}'$ by the hypothesis, and $\mathfrak{b}/\mathfrak{n}'$ because so is $\mathfrak{b}/\mathfrak{e}$  \cite[Lemma 3.5]{ArAnF}. Thus all the hypotheses in Proposition~\ref{expVN} hold and this completes the proof.
\end{proof}

Fix word length functions $\ell_1$ and $\ell_2$ on $B/N'$ and $L$, respectively. Let $\tau\!:B\to B/N'$ be the quotient homomorphism.
Put
\begin{equation}\label{thredecw}
\varphi(g)\!:= \log(1+|\eta|)+\ell_1(\tau(\exp(\xi)))+\ell_2(l),
\end{equation}
where $g$ is uniquely decomposed as $\exp(\eta)\exp(\xi)\,l$ with  $\eta\in\mathfrak{n}'$, $\xi\in\mathfrak{v}$ and $l\in L$ (see Theorem~\ref{expVNL}) and $|\cdot|$ is a norm on~$\mathfrak{n}'$.

We now state our main result. (See a partial case in \cite[Theorem~4.1]{ArAnF}.)

\begin{thm}\label{maindisto}
Let $G$ be a connected linear complex Lie group, $N'$ a normal integral subgroup of $G$ such that $E\subset N'\subset N$ and $\varphi$ the function defined in~\eqref{thredecw}.

\emph{(A)}~Then every length function exponentially distorted on $N'$ is dominated by $\varphi$.

\emph{(B)}~There is a length function strictly exponentially distorted on $N'$, equivalent to
$\varphi$, and separating the points of~$G$.
\end{thm}

Thus there is a maximal equivalence class of length functions exponentially distorted on $N'$ and this motivates us to introduce  the following notion.

\begin{df}
Let $G$ be a connected linear complex Lie group  and $N'$ a normal integral subgroup of $G$ such that $E\subset N'\subset N$. If $\ell'$ is a length function on $G$ exponentially distorted on $N'$ such that $\ell\lesssim \ell'$ for every length function $\ell$ exponentially distorted on $N'$, then we say that $\ell'$  is a \emph{maximal length function exponentially distorted on~$N'$}.
\end{df}
Thus Theorem~\ref{maindisto} states that a maximal length function exponentially distorted on~$N'$ does exist and is equivalent to the function $\varphi$ defined in~\eqref{thredecw}.

To prove the theorem we need the following technical lemma.

\begin{lm}\label{ellell1pr}
Let $\ell$  be a length function on $G$ that is exponentially distorted on $N'$ and $\ell_1$ a  word length function on $G/N'$. Then $\ell\simeq \ell_1\circ\tau$ on $\exp\mathfrak{v}$.
\end{lm}
\begin{proof}
The argument is the same as in \cite[Lemma~4.9]{ArAnF} (which in turn is a modification of that in \cite[Lemma~5.2]{Co08}). The only difference is that in \cite{ArAnF} the partial case when $N'=E$ is considered and the fact that a word length function on~$E$ is always exponentially distorted is used.
\end{proof}

\begin{proof}[Proof of Theorem~\ref{maindisto}]
(A)~Let $\ell$ be a length function exponentially distorted on $N'$. Using Theorem~\ref{expVNL} we can
write every $g\in G$ as
$g=\exp(\eta)\exp(\xi)l$, where $(\eta,\xi,l)\in\mathfrak{n}'\times \mathfrak{v}\times L$. Then
$$
\ell(\exp(\eta)\exp(\xi)l)\le \ell(\exp(\eta))+\ell(\exp(\xi))+\ell(l).
$$
It suffices to estimate every summand on the right by the corresponding summand in~\eqref{thredecw}.

First, Lemma~\ref{closc} implies that $N'$ is nilpotent, closed and simply connected. So by Part~(A) of Proposition~\ref{charexd}, we have that $\ell(\exp(\eta))\lesssim \log(1+|\eta|)$ on $\mathfrak{n}'$.
Second, Lemma~\ref{ellell1pr} implies that $\ell\lesssim \ell_1\tau$ on $\exp\mathfrak{v}$ (here $\ell_1$ is a word length function on $B/N'$). To estimate the third summand note that each length function is dominated by a word length function. Applying this to $L$ we get $\ell(l)\lesssim \ell_2(l)$.
The proof of Part~(A) is complete.

(B)~Being linear, $G$ admits a faithful finite-dimensional holomorphic representation~$\pi$, which we can treat as a homomorphism to a finite-dimensional Banach algebra. Let $\sigma$ denote the quotient homomorphism $G\to G/N'$ and $\tilde\ell$ a symmetric word length function on $G/N'$. Note that $\tilde\ell\circ\sigma$ is a length function on~$G$ and it is exponentially distorted on $N'$ being equal to~$0$ on it. Now put
\begin{equation}\label{allprinP}
\ell'\!:=\tilde\ell\circ\sigma+\ell_\pi^{sym},
\end{equation}
where $\ell_\pi^{sym}$ is defined as in~\eqref{lfhom}.

Since a finite-dimensional algebra is a PI-algebra (see, e.g., \cite[p.\,21, Corollary 1.2.25]{KKR16}), it follows from Proposition~\ref{omegpr0} that $\ell'$ is strictly exponentially distorted on $N'$.
By Part~(A), we have $\ell'\lesssim \varphi$. To complete the proof we need to show that $\varphi\lesssim \ell'$.

First we claim that
\begin{equation}\label{ellpexet}
\ell'(\exp(\eta))+ \tilde\ell(\sigma(x)) \lesssim  \ell'(\exp(\eta)x)\quad\text{on $\mathfrak{n}'\times X$,}
\end{equation}
where $X\!:=(\exp \mathfrak{v})L$. Indeed, since both $\tilde\ell$ and $\ell'$ are symmetric, it follows from \cite[Lemma~4.5]{ArAnF} that it suffices to show that
\begin{equation}\label{2cond}
\ell'\lesssim \tilde\ell\circ\sigma\quad \text{on $X$} \quad \text{and} \quad\tilde\ell\circ\sigma\lesssim \ell'\quad\text{on $G$}.
\end{equation}

The second condition holds by the definition of $\ell'$. To verify the first condition note that by Lemma~\ref{ellell1pr}, $\ell'\lesssim \ell_1\circ\tau$ on $\exp\mathfrak{v}$. On the other hand,
$\ell'\lesssim \ell_2$ on $L$ because $\ell_2$ a word length function (here
$\ell_1$ and $\ell_2$ are word length functions on $B/N'$ and $L$, respectively). Therefore,
\begin{equation}\label{ellpvL}
\ell'(\exp(\xi)l)\le \ell'(\exp(\xi))+\ell'(l) \lesssim \ell_1(\tau(\exp(\xi)))+\ell_2(l) \quad\text{on $\mathfrak{v}\times L$.}
\end{equation}

We have by \cite[Theorem~3.14]{ArAnF} that $G/E\cong B/E\times L$. Since $E\subset N'\subset B$, we obtain $G/N'\cong B/N'\times L$. Being a word length function on $G/N'$, the function $\tilde\ell$ is equivalent to the sum of word length functions.
In particular,
\begin{equation}\label{ell12ti}
\ell_1(\tau(\exp(\xi)))+\ell_2(l)\lesssim
\tilde\ell(\sigma((\exp(\xi)l))
\quad\text{on $\mathfrak{v}\times L$}.
\end{equation}
Combining~\eqref{ellpvL} and~\eqref{ell12ti}, we have the first condition in~\eqref{2cond}  and so the claim is proved.

Further, it follows from~\eqref{ell12ti} and~\eqref{ellpexet} that
$$
\ell'(\exp(\eta))+\ell_1(\tau(\exp(\xi)))+\ell_2(l)\lesssim
\ell'(\exp(\eta)\exp(\xi)l)
\quad\text{on $\mathfrak{n}'\times\mathfrak{v}\times L$}.
$$
Since $\ell'$ is exponentially distorted on $N'$, Proposition~\ref{charexd} implies that $\ell'(\exp(\eta))\simeq \log (1+|\eta|)$ on $\mathfrak{n}'$. Thus $\varphi\lesssim \ell'$.

Finally, note that a word length function always separates points. Thus being composition of $G\to G/N'$ and a word length function, $\tilde\ell\circ\sigma$ separates the cosets of $N'$. Also $\ell'$ separates the points of $N'$  being strictly exponentially distorted on it. Hence $\ell'$ separates the points of $G$.
\end{proof}

\begin{exm}
Recall that an $n$-dimensional nilpotent Lie algebra $\mathfrak{g}$ is called \emph{filiform} if the dimension of the $i$th term in the lower central series is equal to $n-i-1$. A standard example is the Lie algebra with basis $e_0,\ldots,e_{n-1}$ and multiplication defined by
$$
[e_0,e_j]=e_{j+1},\qquad (i=1,\ldots, n-2)
$$
and the undefined brackets being zero. See details in \cite[pp.~40--42]{GK96}.

Let $\mathfrak{g}$ be a filiform Lie algebra and $G$ the associated simply connected Lie group. Fix an $\mathscr{F}$-basis (i.e., compatible with the lower central series) in $\mathfrak{g}$ and consider the corresponding the canonical coordinates $t_1,\ldots,t_n$ of the first kind on $G$. Then a word length function on $G$ is equivalent to the function
$$
|t_1|+|t_2|+|t_3|^{1/2}+\cdots+|t_n|^{1/(n-1)};
$$
see \cite{ArAMN}, Theorem 3.1  (cf. Example 3.7 there, where the algebra is also filiform).

Let $k\in\{2,\ldots,n-1\}$ and put $H_k\!:=\{g\in G\!:\,t_1=\cdots=t_k=0\}$. Since $G$ is nilpotent, $E=\{1\}$. On the other hand, it is easy to see that $N=H_2$. Applying Theorem~\ref{maindisto}, we conclude that the maximal length function exponentially distorted on~$H_k$ exists and is equivalent to
$$
\log(1+|t_1|+\cdots+|t_k|)+|t_{k+1}|^{1/k}+\cdots+|t_n|^{1/(n-1)}.
$$
\end{exm}

\begin{pr}\label{N1N2}
 Let $G$ be a connected linear complex Lie group, $N_j$ a normal integral subgroup of $G$ such that $E\subset N_j\subset N$, and $\ell_j$ a maximal length function exponentially distorted on~$N_j$, where $j=1,2$. Then $\ell_2 \lesssim \ell_1$ if and only if $N_1\subset N_2$. In particular,  $\ell_2 \simeq \ell_1$ if and only if $N_2=N_1$.
\end{pr}
\begin{proof}
(1)~Suppose that $N_1\subset N_2$. Let $\tilde\ell_1$ and $\tilde\ell_2$ be word length functions on~$N_1$ and $N_2$, respectively. Then $\tilde\ell_2\lesssim\tilde\ell_1$ on $N_1$.

We claim that then $\log(1+ \tilde{\ell_2}\,)\lesssim \log(1+ \tilde{\ell_1}\,)$ on $N_1$. Indeed, take  $C,D>0$ such that
$\tilde\ell_2(g)\le C\,\tilde\ell_1(g)+D$ for all $g\in N_1$. Fix $g$ and note that
$$
\log(1+ \tilde\ell_2(g))\le \log(1+ C\,\tilde\ell_1(g)+D)\le \log(1+D)+\log(1+ k\,\tilde\ell_1(g)),
$$
where $k$ is an integer satisfying $k\ge (1+D)^{-1}C$. Since $1+ kt\le (1+t)^k$ for every $t\ge0$, we conclude that $\log(1+ \tilde{\ell_2}\,)\lesssim \log(1+ \tilde{\ell_1}\,)$ on $N_1$.

Further, $\ell_2 \lesssim  \log(1+ \tilde\ell_2\,)$  on~$N_2$ since $\ell_2$ is exponentially distorted on~$N_2$. Hence $\ell_2 \lesssim  \log(1+ \tilde\ell_1\,)$ on $N_1$, i.e.,  $\ell_2$ is exponentially distorted on~$N_1$. So being a maximal length function exponentially distorted on~$N_1$, the function $\ell_1$ dominates $\ell_2$ on the whole~$G$.

(2)~Suppose that $N_1$ is not contained in $N_2$, i.e., $N_1\ne N_1\cap N_2$. Denote the Lie algebra of $N_j$, $j=1,2$, by $\mathfrak{n}_j$. Since $N_j$ is nilpotent, closed and simply connected (Lemma~\ref{closc}), we have that $\mathfrak{n}_1\ne \mathfrak{n}_1\cap \mathfrak{n}_2$. Then there is $\eta \in \mathfrak{n}_1\setminus(\mathfrak{n}_1\cap \mathfrak{n}_2)$.

Let $\mathfrak{b}$ and $\mathfrak{h}$ be as above. We claim that $\eta$  can be taken in a subspace $\mathfrak{v}_2$ of~$\mathfrak{h}$ complementary to $\mathfrak{n}_2\cap \mathfrak{h}$. Indeed,
since $\mathfrak{e}\subset \mathfrak{n}_1\cap \mathfrak{n}_2$, we have $(\mathfrak{n}_1\cap \mathfrak{n}_2)+\mathfrak{h}=\mathfrak{b}$ (cf. the proof of Theorem~\ref{expVNL}).
Evidently $\eta\in\mathfrak{b}$ and so there are $\eta'\in\mathfrak{n}_1\cap \mathfrak{n}_2$ and $\eta''\in\mathfrak{h}$ such that $\eta=\eta'+\eta''$. Then  $\eta'\in \mathfrak{n}_1$ and also $\eta''\notin \mathfrak{n}_2$ because otherwise $\eta\in\mathfrak{n}_2$. Since $\eta''$ is in $\mathfrak{h}$ but not in $\mathfrak{n}_2\cap \mathfrak{h}$, we can take $\mathfrak{v}_2$ containing $\eta''$. The claim is proved.

Note that $\mathfrak{b}/\mathfrak{n}_2$ is nilpotent. It is not hard to see (see the proof of Proposition~4.1 in \cite{ArAMN}) that there is an $\mathscr{F}$-basis (i.e., compatible with the lower central series) $(e_k)$ in $\mathfrak{b}/\mathfrak{n}_2$  containing the element $\eta+\mathfrak{n}_2$.
By Theorem~\ref{maindisto}, $\ell_j$, $j=1,2$, is equivalent to the function defined in~\eqref{thredecw} with $N'=N_j$. To use this equivalence we need the following asymptotic for a word length functions on $B/N_j$: it is equivalent to the function  $\max_k|t_k|^{1/w_k}$, where $(t_k)$ are the canonical coordinated of the first kind associated with the $\mathscr{F}$-basis and $w_k\in\mathbb{N}$, \cite[Theorem 3.1]{ArAMN}. Note that the quotient map $\mathfrak{b}\to\mathfrak{b}/\mathfrak{n}_2$ induces a bijection $\mathfrak{b}\to\mathfrak{v}_2$. Putting $j=2$ and taking $k$ such that $e_k=\eta+\mathfrak{n}_2$ we have from~\eqref{thredecw} that
$$
\ell_2(\exp (t \eta)) \simeq |t|^{1/w_k},$$
where $t\in\mathbb{C}$. On the other hand, $\eta\in\mathfrak{n}_1$ and so \eqref{thredecw} written for $j=1$ gives
$$
\ell_1(\exp (t \eta)) \simeq \log(1+|t\eta|),
$$
where $|\cdot|$ a norm on~$\mathfrak{n}$. Thus $\ell_2$  is not dominated by $\ell_1$.
\end{proof}

We finish with remarks.

\begin{rem}\label{nqifam}
If a length function $\ell$ is symmetric and separates points, then $d(g,h)=\ell(g^{-1}h)$ determines a left-invariant distance. Thus, under the hypotheses of Theorem~\ref{maindisto} we obtain the left-invariant distance associated with a maximal length function exponentially distorted on~$N'$. Hence there is a family of equivalence classes (up to quasi-isometry) of metric spaces parameterized by integral subgroups $N'$ such that $E\subset N'\subset N$. By Proposition~\ref{N1N2}, representatives of two different classes are not pairwise quasi-isometric. The family has two extreme cases: $N'=E$ (the biggest) and $N'=N$ (the smallest). The biggest class contains all the distances associated with word length functions on $G$  and also with left invariant Riemannian metrics; see, e.g., \cite[Proposition 5.5]{ArAMN}). So in the case when $E\ne N$  we have left-invariant structures of metric space  on $G$ that do not correspond to Riemannian metrics.

Of course, we always can replace a Riemannian distance by a non-Riemannian by composing it with a subadditive increasing continuous function with sufficiently low growth (see Lemma~\ref{3oper}(1)). But the distances in our family are more interesting because they satisfies the maximality property.
\end{rem}

\begin{rem}\label{nonscreal}
It follows from Lemma~\ref{closc} that the subgroup~$N'$ in Theorem~\ref{maindisto} is simply connected and nilpotent. Now we give a brief description of construction of subgroups of the nilpotent radical  that are not simply connected but nevertheless admit an exponentially distorted length function.

Let $H$ be the Heinsenberg group as in Example~\ref{Heisex}. Recall that the nilpotent radical $N$ coincides with the centre and is isomorphic to~$\mathbb{C}$. Fix a copy of $\mathbb{Z}$ in $\mathbb{C}$ and put $G=H/\mathbb{Z}$. We claim that $G$  admits  a length function exponentially distorted on the subgroup $\mathbb{C}^\times=N/\mathbb{Z}$, which coincides with the nilpotent radical of~$G$.

It is well known that $G$ is not linear and so we cannot apply to it our results on complex Lie groups. But the notion of an exponentially distorted length function does not rely to Lie group structure. In particular, we can treat $G$ as a real Lie group.

Since $\mathbb{C}^\times\cong\mathbb{R}\times \mathbb{T}$, we can consider the quotient $G/\mathbb{T}$, which is $\mathbb{R}$-linear being nilpotent and simply connected \cite[p.\.595, Proposition 16.2.6]{HiNe}. By a version of Proposition~\ref{omegpr0} for real Lie groups, there exists a length function exponentially distorted on the copy of $\mathbb{R}$ contained in $\mathbb{C}^\times$. (The proof of Proposition~\ref{omegpr0} based on results in \cite{ArPiLie}, which hold not only for $\mathbb{C}$ but also for $\mathbb{R}$. We even do need the real version of \cite[Theorem 2.2]{ArAnF} stated in~\cite{Ho} because use here a finite-dimensional representation.) Therefore this length function is (not strictly) exponentially distorted on~$\mathbb{C}^\times$.

In fact, this construction works for an arbitrary non-linear connected Stein group~$G$. Recall that the \emph{linearizer} of a complex Lie group is the intersection of the kernels of all holomorphic representations. So it is contained in the nilpotent radical.  By \cite[Corollary~2]{ArLine}, the linearizer contains a copy of $\mathbb{C}^\times$ as a direct factor. Consider the case when the linearizer is isomorphic to $\mathbb{C}^\times$. For example, this holds for the Heisenberg group. (It is worth mentioning that it can be easy to deduce from \cite[Proposition 5.1]{Ar20+} that any distorted copy of $\mathbb{C}^\times$ is contained in the linearizer.) Then there is a length function exponentially distorted on it. It holds by the same argument  as above with the only difference that we need to show that $G/\mathbb{T}$ is $\mathbb{R}$-linear. This can be done as in \cite[Theorem~1]{ArLine} using the fact that $G/\mathbb{T}$ is an extension of an $\mathbb{R}$-linear group by $\mathbb{R}$. The general case can be reduced to  this by factorization.
\end{rem}

\end{document}